\documentclass[12pt,a4paper]{article}

\usepackage{latexsym,amssymb,amsthm}
\usepackage{url,graphics}
\usepackage[dvips]{graphicx,epsfig}

\title{\bf  Spanning trees with many leaves:  lower bounds in terms of number of vertices of degree 1, 3 and at least~4}
\author{
         D.\,V.\,Karpov\thanks{This research was supported by Russian Foundation for Basic Research (RFBR) grant 11-01-00760-a.}
     \\[6pt]
          {\small E-mail: \texttt{dvk0@yandex.ru}       }
    }
\date{}

\begin{document}
\maketitle
\righthyphenmin=2
\renewcommand*{\proofname}{\bf Proof}
\newtheorem{thm}{Theorem}
\newtheorem{lem}{Lemma}
\newtheorem{cor}{Corollary}
\theoremstyle{definition}
\newtheorem{defin}{Definition}
\theoremstyle{remark}
\newtheorem{rem}{\bf Remark}

\def\N{{\rm N}}
\def\q#1.{{\bf #1.}}
\def\I{{\rm Int}}
\def\R{{\rm Bound}}
\def\mmin{\mathop{\rm min}}

\centerline{\sc Abstract}

We prove that every connected graph with $s$ vertices of degree~1 and~3 and~$t$ vertices of degree at least~4
has a spanning tree  with  at least ${1\over 3}t +{1\over 4}s+{3\over 2}$ leaves.
We present infinite series of graphs showing that our bound is tight.

\section{\bf Introduction. Basic notations}

We consider undirected graphs without loops and multiple edges and use standard notations. For a graph $G$ we denote the set of its vertices by  $V(G)$ and the set of its edges by $E(G)$. We use notations~$v(G)$ and~$e(G)$ for the number of vertices and edges of $G$, respectively.

For any set $E\subset E(G)$ we denote by~${G-E}$ the graph obtained from~$G$ after deleting edges of the set~$E$. For any set~$V\subset V(G)$ 
we denote by~$G-V$ the graph obtained from~$G$ after deleting vertices of the set~$V$ and all edges incident to deleted vertices.

As usual, we denote the {\it degree} of a vertex $x$ in the graph $G$ by $d_G(x)$. We denote the minimal vertex degree of the graph~$G$ by~$\delta(G)$. 
Let $\N_G(w)$ denote the {\it neighborhood}  of a vertex  $w\in V(G)$ (i.e. the set of all vertices of the graph~$G$, adjacent to~$w$).

For any edge $e\in E(G)$ we denote by $G\cdot e$ the graph, in which the ends of the edge $e=xy$ are {\it contracted} into one vertex, which is adjacent  to all different from~$x$ and~$y$ vertices, adjacent in~$G$ to at least one of the vertices~$x$ and~$y$. Let us say that the graph $G\cdot e$ is obtained from $G$ by {\it contracting} the edge $e$.

\begin{defin}
For any connected graph~$G$ we denote by~$u(G)$  the maximal number of leaves in a spanning tree of the graph~$G$. 
\end{defin}

\begin{rem}
Obviously, if  $F$ is a tree, then   $u(F)$ is the number of its leaves.
\end{rem}

Several papers about lower bounds on~$u(G)$ are published.  
One can see details of the history of this question in~\cite{KB}. We shall recall only results, directly concerned with our paper.

In 1981 N.\,Linial formulated a  conjecture:  
$$u(G)\ge {d-2\over d+1}v(G) + c \quad \mbox{as} \quad \delta(G)\ge d\ge 3,$$
 where a constant $c>0$ depends only on~$d$. The  ground  for this  conjecture is the  following:  for every  $d\ge 3$
one can easily construct infinite series of graphs with   minimal degree  $d$, for 
which  ${u(G)\over v(G)}$ tends to $d-2\over d+1$.

It follows from the papers~\cite{Alon,DJ,YW} that for $d$ large enough  Linial's conjecture fails.  However, we are interested in the case of  small~$d$.
 
In 1991 Kleitman and West~\cite{KW} proved, that $u(G)\ge {1\over 4}\cdot v(G)+2$ as~$\delta(G)\ge 3$ and~$u(G)\ge {2\over 5}\cdot v(G)+{8\over 5}$ as~$\delta(G)\ge 4$.
In~1996 Griggs and Wu~\cite{JGM} once again proved the statement for~$\delta(G)\ge4$
and proved, that~$u(G)\ge {1\over 2}\cdot v(G)+2$ as~$\delta(G)\ge 5$. 
Hence, Linial's conjecture holds for~$d=3$, $d=4$ and~$d=5$, for~$d>5$ the question remains open.

In~\cite{KW} a more strong Linial's conjecture was mentioned: 
   $$u(G)\ge \sum_{x\in V(G)} {d_G(x)-2\over d_G(x)+1}$$ 
for a connected graph~$G$ with $\delta(G)\ge 2$. 
Clearly, this conjecture is not true, since weak Linial's conjecture fails for large degrees.
However, strong Linial's  conjecture inspires attempts to obtain a lower bound on~$u(G)$, in which contribution of each vertex  depends on its degree. 

In~\cite{K2} D.\,V.\,Karpov has proved, that for a connected graph~$G$ with ${v(G)>1}$, $s$ vertices of degree~$3$ and~$t$ vertices of degree at least~$4$ the inequality  $u(G)\ge {2\over 5 } t + {1\over 5} s +2$ holds besides three graphs-exclusions.
In this paper vertices of degree~1 and~2 are allowed in the graph. Infinite series of examples show us that  this bound is tight.

In~\cite{KB, Ba} lower bounds on~$u(G)$ taking into account the vertices of degree~1 of the graph~$G$ were proved.
Let~$G$ be a connected graph with~$s$ vertices of degree not~2 and~$v(G)>1$.
In~\cite{KB} D.\,V.\,Karpov and A.\,V.\,Bankevich have proved that $u(G)\ge {1\over 4}s +{3\over 2}$.
In~\cite{Ba} A.\,V.\,Bankevich has proved that 
$u(G) \ge {1\over 3} s  + {4\over 3}$  for a triangle-free graph~$G$. All these bounds are tight, that is confirmed by infinite series
of examples.

 It is interesting to look  at our main result ~--- theorem~\ref{u134}~--- on this background.

\begin{thm}
\label{u134} 
Let~$G$ be a connected graph with~$s$ vertices of degree~$1$ and~$3$, ~$t$ vertices of degree at least~$4$ and~$v(G)>1$. 
Then $u(G)\ge {1\over 3} t + {1\over 4} s + {3\over2}$. 
\end{thm}

Note, that  all three constants in this bound are optimal. There are graphs for which this bound is tight. 
In the end of this paper we  present infinite series of such graphs, containing only vertices of degrees~1, 3 and~4. 

Our proof uses several methods of construction of a spanning tree with many leaves. 
We use reduction techniques, developed in~\cite{KB}. In the case~$R6$ of the proof we use in reduction
the techniques of deletion of edges of some route, developed by A.\,V.\,Bankevich in~\cite{Ba}. 
However, these methods are significantly modified to take  account of vertices of degree at least~4 separately from
vertices of less degrees. Unfortunately, we cannot prove our theorem only by reduction techniques as in~\cite{KB,Ba}.

In the cases that remain after applying reduction we use the classic method of {\it dead vertices}~\cite{KW, JGM, K2}. 
Since this method does not allow to take into account the vertices of degree~1, we must include all pendant vertices of~$G$ in the base of construction.  Due to this our base construction will be a forest, not a tree as in~\cite{KW, JGM, K2}. 
One can see such idea in the paper~\cite{Gr} by  N.\,V.\,Gravin, but in quite different context.

\section{Proof of Theorem~1}

As usual we assume, that the statement of our theorem is proved for all {\it less graphs} (i.e. graphs with less number of vertices, than $G$, or with the same number of vertices and less number of edges).

Let~$S(G)$ be the set of all vertices of degrees~1 and~3 and~$T(G)$ be the set of all vertices of degree at least~4 in the graph~$G$, 
$s(G)=|S(G)|$,  $t(G)=|T(G)|$.

\begin{defin}
We call by the {\it cost} of a vertex~$x$ in the graph~$G$  the following value~$c_G(x)$: 
$$ c_G(x)= \left\{
\begin{array}{ll}
                {1\over 3}, & \quad \mbox{as} \quad x\in T(G), \\[2pt]
                 {1\over 4},  &\quad \mbox{as} \quad x\in S(G), \\
                 0, &\quad \mbox{otherwise.}
\end{array} \right.
$$

Let $$c(G)={1\over 3} t(G) +{1\over 4}s(G)=\sum_{x\in V(G)} c_G(x) $$
be the {\it cost} of the graph~$G$. 

For any subgraph~$F$ of the graph~$G$ we define its  {\it cost in the graph}~$G$ by
$$c_G(F)=\sum_{x\in V(F)} c_G(x).$$
\end{defin}

We desire to proof the inequality~$u(G)\ge c(G)+{3\over 2}$, which is equivalent to the statement our theorem.

\subsection{ Reduction}

In some cases we will reduce the problem for our graph~$G$ to the same problem for a less graph.
This part is based on the same ideas as in~\cite{KB} with the only difference: now we need to take account of vertices of degree at least~4 separately from cheaper  vertices of degrees~1 and~3.

Let us formulate a definition and two lemmas from the paper~\cite{KB}.

\begin{defin}   
Let  $G_1$ and $G_2$ be two graphs with marked vertices $x_1\in V(G_1)$ and $x_2\in V(G_2)$ respectively, $V(G_1)\cap V(G_2)=\varnothing$. 
To {\it glue} the graphs $G_1$ and $G_2$ by the vertices  $x_1$ and~$x_2$ is to glue together vertices $x_1$ and $x_2$ into one vertex  $x$,  which will be incident to all edges, incident to  $x_1$ or $x_2$ in both graphs 
$G_1$ and $G_2$. All other vertices and edges of the graphs  $G_1$ and $G_2$  will be vertices and edges of the resulting graph.
\end{defin}

\begin{lem}
\label{tool} 
 Let $G_1$ and $G_2$ be connected graphs with  $V(G_1)\cap V(G_2)=\varnothing$, $v(G_1)\ge 2$, 
$v(G_2)\ge 2$  and pendant vertices~$x_1$ and~$x_2$, respectively. Let~$G$ be a graph, obtained  by gluing 
$G_1$ and $G_2$ by the vertices  $x_1$ and $x_2$  and, after that, by contracting  several bridges, not incident to pendant vertices.  Then  $u(G)=u(G_1)+u(G_2)-2$.
\end{lem}

\begin{lem}
\label{tr1}
Let $a,b\in V(G)$ be adjacent vertices, and subgraph $G'$ be a connected component
of the graph $G-a$, which contains the vertex $b$. Let  $b$ be a cutpoint of the graph $G'$.
Then  $u(G)\ge u(G')+1$.
\end{lem}

{\sf Let us describe Reduction rules. We consider several cases. Considering each case we assume that  condition of any previous case 
does not hold.}

\smallskip
\q{R1}. {\it The graph~$G$ has a  vertex~$a$ of degree~$2$.}

\noindent  
 If~$a$ is a cutpoint then we contract an edge incident to~$a$ and obtain less graph~$G'$ with~$c(G')=c(G)$ and~$u(G')=u(G)$.

If~$a$ is not a cutpoint then  an edge~$ab$ incident to~$a$ is not a bridge and the graph~$G'= G-ab$ is connected. 
Clearly, $c_{G'}(a)-c_{G}(a)={1\over 4}$ and~$c_{G}(b)-c_{G'}(b)\le {1\over 4}$, hence $c(G')\ge c(G)$. 
Since any spanning tree of the graph~$G'$ is a spanning tree of~$G$, then we have $u(G')\le u(G)$. 

In both cases the statement for the graph~$G$ follows from the statement for less graph~$G'$.

\begin{rem}
 In what follows we assume that the graph~$G$ has no vertices of degree~$2$. 
\end{rem}

\smallskip
Let~$U$  be the set of all pendant vertices of the graph~$G$. In this section we assume that $U\ne\varnothing$. We will consider the case when the graph~$G$ has no pendant vertices later.

If  two vertices of~$U$  are adjacent, then the graph contains only these two vertices, for this graph  and the statement of theorem~\ref{u134} is obvious. {\sf Further we consider graphs with more than two vertices, hence,  any two   vertices  of the set~$U$ are  not adjacent.} 

Let  $W\subset V(G)$ be the set of all vertices,  adjacent to pendant vertices. Let~$X\subset V(G)$ be the set of all vertices, adjacent to~$W$, that does not  belong to~$U\cup W$.

Let $H=G-U$. Clearly, the graph~$H$ is connected.

\smallskip
\q{R2}. {\it The graph~$H$ is not biconnected.}

\noindent 
Let~$a$ be a cutpoint of the graph~$H$.   Then $a$ is a cutpoint of the graph~$G$ and there exist connected graphs~$G_1$ and~$G_2$ such that
$$V(G_1)\cup V(G_2)=V(G), \quad V(G_1)\cap V(G_2)=\{a\} \quad \mbox{and} \quad  v(G_1), v(G_2) > 2.$$
For~$i\in\{1,2\}$ we define a graph~$G_i'$, obtained from~$G_i$ as follows: we adjoin a new pendant vertex~$x_i$  to the vertex~$a$ (see fig.~\ref{figa1}).  

Let us glue~$G'_1$ and~$G'_2$ by the vertices~$x_1$ and~$x_2$ (they form new vertex~$x$) and, after that, contract two bridges, incident to~$x$. We obtain the graph~$G$ as a result of these operations.  Note, that two copies of the vertex~$a$ in the graphs~$G'_1$ and~$G_2'$
are contracted into the vertex~$a$ of the graph~$G$. Thus the graphs~$G$, $G_1'$ and~$G_2'$ satisfy all conditions of lemma~\ref{tool} and we have~$u(G)=u(G_1')+u(G_2')-2$.

Since~$d_G(a)=d_{G'_1}(a)+d_{G'_2}(a)-2$, it is easy to see, that~$c_{G}(a)\le c_{G'_1}(a)+ c_{G'_2}(a)$. 
The vertices~$x_1\in V(G_1')$ and~$x_2\in V(G_2')$ does not belong to~$V(G)$, we have~$c_{G_1'}(x_1)=c_{G_2'}(x_2)={1\over 4}$. 
Every vertex of the graph~$G$ except~$a$ belongs to exactly one of the graphs~$G'_1$ and~$G'_2$ and have in this graph the same degree as in the graph~$G$. Hence it follows that~$c(G)\le c(G'_1)+ c(G'_2) -2\cdot {1\over 4}$.

\begin{figure}[!hb]
	\centering
		\includegraphics[width=\columnwidth, keepaspectratio]{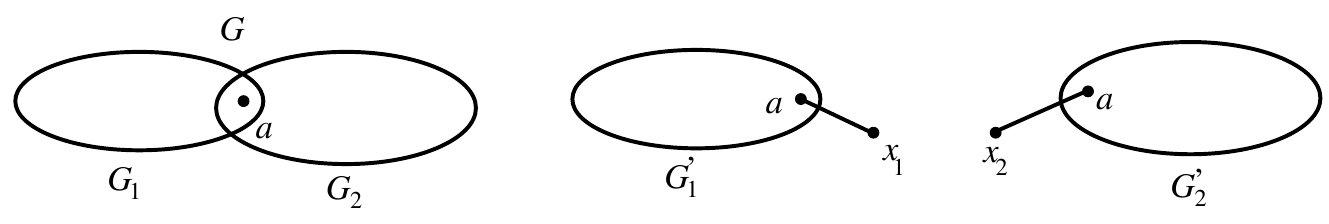}
     \caption{Reduction $R2$}
	\label{figa1}
\end{figure} 

Clearly,    $v(G_1')<v(G)$ and $v(G_2')<v(G)$. Then by induction assumption we have
$u(G'_1)\ge c(G'_1)+{3\over2}$ and  $u(G'_2)\ge c(G'_2)+{3\over2}$. 
Thus, 
$$u(G)=u(G_1')+u(G_2')-2 \ge c(G'_1)+ c(G'_2) -{1\over 2} + {3\over2} \ge c(G) + {3\over2},$$
what was to be proved.

\begin{rem}
In what follows we assume that the graph~$H$ is biconnected. Hence all cutpoints of the graph~$G$ are vertices of the set~$W$. 
Each vertex of the set~$W$ separates pendant vertices adjacent to it from all other vertices of the graph~$G$.
\end{rem}

\smallskip \goodbreak
\q{R3}. {\it There exist adjacent vertices~$x,y\in V(G)$, such that
${d_G(x)\ge 5}$ and~${d_G(y)\ge 5}$.}

\noindent Let us consider the graph~$G'=G-xy$. Since the graph~$H=G-U$ is biconnected then the graph~$G'$ is connected and less than~$G$. Hence the statement of theorem is proved for~$G'$. Clearly, ${d_{G'}(x)\ge 4}$ and~${d_{G'}(y)\ge 4}$,  hence~$c_{G}(x)=c_{G'}(x)$ and~$c_{G}(y)=c_{G'}(y)$.
Consequently,  $c(G')=c(G)$. Since any spanning tree of the graph~$G'$ is a spanning tree of the graph~$G$,  the statement of theorem~\ref{u134} is proved for~$G$.

\smallskip
\q{R4}. {\it There exist adjacent vertices~$a,b\in V(G)$, such that~$b$ is a cutpoint of a connected component~$G'$ of the graph~${G-a}$ and $c(G')\ge c(G)-1$.}

\noindent  Since $b\in \N_{G'}(a)$  is a cutpoint of the graph~$G'$, by lemma~\ref{tr1} we have~$u(G)\ge u(G')+1$. The statement is proved for less connected graph~$G'$,  hence  $$u(G)\ge u(G')+1 \ge c(G')+1+{3\over 2}\ge c(G)+{3\over 2},$$
what was to be proved.

\begin{lem}
 \label{lr4}
If the graph~$G$ satisfies one of the following conditions, then the reduction~$R4$ can be done.

$1^\circ$ There  exist adjacent vertices~$a,b\in V(G)$, such that~$b$ is a cutpoint of a connected component~$G'$ of the graph~$G-a$ and~$d_G(a)\le 3$.

$2^\circ$ There   exist adjacent vertices $w\in W$  and~$x\not \in U$ such that ${d_G(x)= 3}$.

$3^\circ$ There   exist two vertices  $x,y\in U$ adjacent  to a vertex~$w\in W$.

$4^\circ$ There   exists a vertex~$x$ adjacent to~$W$  such that ${d_G(x)\le 6}$ and~$x$ is adjacent to not more than one vertex from~$S(G)$.

$5^\circ$ There   exists a vertex~$x$ adjacent to~$W$  such that ${d_G(x)=4}$ and~$x$ is adjacent to not more than two vertices from~$S(G)$.

\end{lem}

\begin{proof}
$1^\circ$  Since the graph~$H$ is biconnected, then~$G'$ contains all vertices of the graph~${G-a}$, except pendant vertices, adjacent to~$a$ (in the case~$a\in W$).  Decreasing the degree of a vertex by~1 we decrease its cost by not more than~$1\over 4$.
Hence $$c(G)-c(G')\le c_G(a)+ \sum_{v\in \N_G(a)} (c_G(v)-c_{G'}(v)) \le {1\over 4}+3\cdot{1\over 4}=1$$
and the condition of~$R4$ for the vertices~$a$ and~$b$ holds.

$2^\circ$ and $3^\circ$.  In both cases we consider the connected component~$G'$ of the graph~$G-x$, which contains~$w$.
Clearly, $w$ is a cutpoint of~$G'$, since there is a a pendant vertex  adjacent to~$w$  different from~$x$ and~$w$ separates this pendant vertex  from other vertices of the graph~$G$. Hence, the graph~$G$ satisfies the condition~$1^\circ$ for~$a=x$ and~$b=w$.

$4^\circ$ and~$5^\circ$. Let~$w\in W$ be a vertex adjacent to~$x$. In both cases we consider the connected component~$G'$ of the graph~$G-x$, which contains~$w$. Clearly,~$G'$ contains all vertices of the graph~$G-x$, except adjacent to~$x$ pendant vertices and~$w$ is a cutpoint of the graph~$G'$. 

It remains to show that~$c(G)-c(G')\le 1$. Decreasing the degree of a vertex~$v$ by~1 we decrease its cost by~$1\over 4$ if~$v\in S(G)$
and by not more than~${1\over 3}-{1\over 4}={1\over 12}$ otherwise. Hence in the case~$4^\circ$ we have:
$$c(G)-c(G')\le c_G(x)+ \sum_{y\in \N_G(x)} (c_G(y)-c_{G'}(y)) \le {1\over 3}+{1\over 4} 
+5\cdot{1\over 12} = 1.$$
Similarly, in the case~$5^\circ$ we have:
$$c(G)-c(G')\le c_G(x)+ \sum_{y\in \N_G(x)} (c_G(y)-c_{G'}(y)) \le {1\over 3}+2\cdot {1\over 4} 
+2\cdot{1\over 12} = 1.$$
Now one can see that the condition of~$R4$ for the vertices~$a=x$ and~$b=w$ holds.
\end{proof}

\begin{rem}
 In what follows we assume that the graph~$G$ does not satisfy any of conditions~$1^\circ-5^\circ$ of lemma~\ref{lr4}.
\end{rem}

\begin{lem}
 \label {lr5}
If no of  reduction rule of $R1-R4$ can be applied, then any two vertices of the set~$W$ are not adjacent in the graph~$G$.

\begin{proof}
Suppose the contrary, let there exist two adjacent vertices~${w,w'\in W}$,  $d_G(w)\le d_G(w')$. 
Then~$d_G(w)\le 4$ (otherwise we can apply~$R3$). Moreover, $d_G(w)=4$ (otherwise~$w$ and~$w'$ satisfy the condition~$2^\circ$ of lemma~\ref{lr4}). All vertices adjacent to~$w$, except one pendant vertex,  have degree at least~4, 
otherwise the graph satisfies one of the conditions~$2^\circ$ and~$3^\circ$ of lemma~\ref{lr4}. But in this case the condition~$4^\circ$ of lemma~\ref{lr4} holds for~$x=w$. We obtain a contradiction.
\end{proof}
\end{lem}

\begin{rem}
\label{rx}
  In what follows any vertex~$w\in W$ is adjacent to one pendant vertex (from~$U$) and~$d_G(w)-1$ vertices of the set~$X$. In particular, whence it follows that~$X\ne\varnothing$. Since the condition~$2^\circ$ of lemma~\ref{lr4} does not hold, all vertices of the set~$X$ have degree at least~4 in the graph~$G$.
\end{rem}

\goodbreak
\q{R5}. {\it There exists a vertex~$x\in X$, adjacent to vertices~$w,w'$, such that~${w\in W}$,   ${d_G(w)=d_G(w')=3}$ and $\N_G(w')$ contains not more than one vertex from~$S(G)$.}

\begin{rem}
\label{rr55}
 Since $X\cap S(G)=\varnothing$, then  in the case~$w'\in W$ the condition~$R5$ holds.
\end{rem}

\noindent Let~$\N_G(w)=\{x,y,u\}$, where~$u\in U$.  It follows from remark~\ref{rx}, that~${d_G(y)\ge 4}$. 
Clearly, the graph~$G'=G\cdot wx$ is connected.  Let the vertex~$x'\in V(G')$ is the result of contracting vertices~$x$ and~$w$ (see figure~\ref{figa2}).

\begin{rem}
 \label{rcp5}
  {\sf Any cutpoint~$v\ne x'$ of the graph~$G'$ is a cutpoint of the graph~$G$.} Moreover, let~$K$ be a connected component of the graph~$G'-v$. If~$K$ does not contain~$x'$, then~$K$ is a connected component of the graph~$G$. If~$K$ contains~$x'$, then~$K\cup\{x,w\} \setminus\{x'\}$ is a connected component of the graph~$G$.
\end{rem}

\begin{figure}[!hb]
	\centering
		\includegraphics[width=0.7\columnwidth, keepaspectratio]{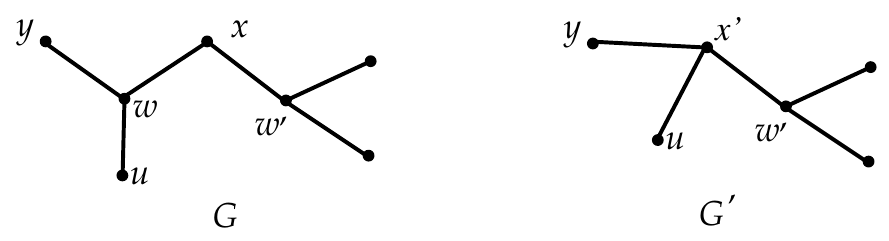}
     \caption{The graphs~$G$ and~$G'$.}
	\label{figa2}
\end{figure}

If~$w'\in W$ then the graph~$G'-w'$ has exactly two connected components, one of them consists of a pendant vertex of the graph~$G'$ adjacent to~$w'$. Let~$G^*$ be another connected component of~$G'-w'$, that contains all other vertices.
If~$w'\not\in W$  then the graph~${G'-w'}$ is connected. In this case let~$G^*=G'-w'$. Clearly, in both cases~$c(G^*)=c(G'-w')$. Since the connected graph~$G^*$ is less then~$G$, for this graph the statement of theorem~\ref{u134} holds.

The vertex~$x'$ is a cutpoint of the graph~$G^*$ (it separates the pendant vertex~$u$), hence by lemma~\ref{tr1} we have
$$u(G)\ge u(G')\ge u(G^*)+1\ge c(G^*)+1 +{3\over 2}.$$ 
It remains to prove that~$c(G^*)=c(G'-w')\ge c(G)-1$. Note, that $d_{G'}(x')\ge 4$, hence $c_{G'}(x')= {1\over 3} = c_{G}(x)$.

After contracting the edge~$xw$ the degree of any vertex~$v$ different from~$x$ and~$w$ is preserved or decreased by one (in the case when~$v$
is adjacent to both vertices~$x$ and~$w$ in the graph~$G$). Hence  only the vertex~$y$ can be a vertex which degree is decreased. Since $y\in X \subset T(G)$ we have $c_{G'}(y)\ge c_{G}(y) -{1\over 12}$. The degrees of any  vertex different from~$x,y,w$ in the graphs~$G$ and~$G'$ coincide.  
Thus  $c(G') \ge c(G) - {1\over 3}$. Since~$d_{G'}(x')\ge4$ then~$S(G)\supset S(G')$.
It remains to note that from $|\N_{G}(w') \cap S(G)|\le 1$ it follows that $|\N_{G'}(w') \cap S(G')|\le 1$. Consequently,
$$c(G'-w')=c(G')-c_{G'}(w')- \sum_{v\in \N_{G'}(w')} (c_{G'}(v)-c_{G'-w'}(v))\ge$$
$$ c(G')- {1\over 4} - \left({1\over 4} + 2\cdot{1\over 12}\right) = c(G')-{2\over 3} \ge c(G)-1, $$
that is enough.

\smallskip
\q{R6}. {\it There exists a vertex~$x\in X$, such that $d_G(x)\le 6$.}

\noindent 
We choose a vertex $x\in X$ of the minimal degree. Since  the condition~$4^\circ$ of lemma~\ref{lr4} does not hold for~$x$, then~$x$ is adjacent to at least two vertices of the set~$S(G)$. Since~$x\not\in W$, these two vertices have degree~3. 
Since we cannot apply~$R5$, at least one of these two vertices does not belong to~$W$ (see remark~\ref{rr55}) and has two neighbors of the set~$S(G)$. Denote this vertex by~$y$, let~$\N_G(y)=\{x,z,z'\}$. Since~$y\notin W$, we have~$d_G(z)=d_G(z')=3$.

Let~$w\in W$  be a vertex adjacent to~$x$.  Consider two cases.

\smallskip
\q{R6.1}. $d_G(x)=4$, $d_G(w)\ge 4$. 

\noindent 
Let $N_G(x)=\{w,y,y_1,y_2\}$. Then  $d_G(y)=d_G(y_1)=d_G(y_2)=3$, since otherwise the condition~$5^\circ$ of lemma~\ref{lr4} holds. 

{\sf Our first aim is to construct in the graph~$G$ a simple path~$P$  from~$y$ to some vertex~$q$  (where $q\not\in S(G)\cup \N_G(x)$ or~$q\in \{y_1,y_2\}$),  such that all inner vertices of the path~$P$ belong to $S(G)\setminus \N_G(x)$ and the graph~$G-E(P)-x$ is connected.}

Consider~$z\in \N_G(y)$, $z\ne x$.  It follows from~$d_G(y)=3$ by remark~\ref{rx} that $y\notin X$, i.e. $y$ cannot be adjacent to the vertex~$w\in W$. Hence $z\ne w$. 
If the edge~$yz$ is a bridge in~${G-x}$, then~$x$ is a cutpoint of the graph~${G-y}$ and the condition~$1^\circ$ of lemma~\ref{lr4} holds for the vertices~$y$ and~$x$. We obtain a contradiction. Then~$yz$ is not a bridge in~${G-x}$.

\begin{figure}[!ht]
	\centering
		\includegraphics[width=0.8\columnwidth, keepaspectratio]{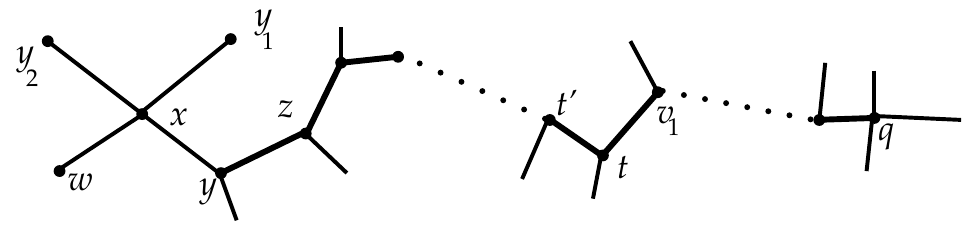}
     \caption{Reduction $R6.1$,  path $P$.}
	\label{figa3}
\end{figure} 

At the beginning let~$P'$ contains only one edge~$yz$.  We know, that the graph~$G-yz-x$ is connected and, hence,~${d_G(z)\ge 3}$.

Let  a path~$P'$ from $y$ to $t$ be built, such that the graph~$G-E(P')-x$ is connected, $d_G(t)\ge 3$  and~$t\ne w$
(at the beginning we have $t=z$ and all these conditions hold). 
Clearly, $t\in S(G)$ (i.e., $d_G(t)=3$) and $t\notin \{y_1,y_2\}$, otherwise the path~$P=P'$ is what we want. 

Let $\N_G(t)=\{t',v_1,v_2\}$, where~$t'$ is the previous vertex of the path~$P'$. Then~$d_G(t')=3$. We try to extend the path~$P'$ by one edge.

Without loss of generality we may assume that~$v_1\ne w$.  Clearly, $v_1\ne x$.
Let~$tv_1$ be a bridge of the graph~$G-E(P')-x$. Then~$t$ is a cutpoint of the graph ${G-E(P')-x-t'}$ (this graph is, obviously, connected, since it is a result of deleting a pendant vertex~$t'$ from the connected graph~$G-E(P')-x$).
Then $$u(G)\ge u(G-E(P')) \ge u(G-E(P')-x-t') +2$$
(we can add two new leaves to a spanning tree of the graph ${G-E(P')-x-t'}$: we adjoin~$x$ to the cutpoint~$w$ and adjoin~$t'$
to the cutpoint~$t$).

Let us estimate  $c(G)-c(G-E(P')-x-t')$. We have deleted from~$G$  the vertices~$x$ and~$t'$ which cost together ${1\over 3}+{1\over 4}$.
Since the cost of vertices of degrees~1 and~3 is the same, then the cost of all vertices of the part~$P'$ besides~$t'$ and~$t$ is the same in both graphs. Thus  deleting  $E(P')$, $x$ and~$t'$   can   decrease the cost  of only 5 vertices: they are three vertices  of ${\N_G(x)\setminus V(P')}$  and two vertices of~$\N_G(t')\setminus V(P')$. Since the cost of each of these 5 vertices is decreased by not more than~$1\over 4$, we have
 $$c(G-E(P')-x-t')\ge c(G)-{1\over 3} - {1\over 4} - 5\cdot {1\over 4}> c(G)-2.$$ It follows from the induction assumption that
$$u(G)\ge u(G-E(P')-x-t') +2 \ge c(G-E(P')-x-t') +2 +{3\over 2}> c(G)+{3\over 2}.$$
In this case the theorem is proved. 

\smallskip
Let us consider the remaining case when~$tv_1$  is not a bridge of the graph ${G-E(P')-x}$. Then~$d_{G-E(P')}(v_1)\ge 2$,
hence, $v_1\not\in V(P')$ and $d_{G}(v_1)\ge 3$.
We extend the path~$P'$ by the edge~$tv_1$ and continue our reasonings with the new path and its end~$v_1$. 
Since our graph is finite, the process will be finished and we obtain the desired path~$P$.

Consider the graph~$G-E(P)-x$. Similarly to proved above, 
$$u(G)\ge u(G-E(P)-x)+1,$$ and it remains to estimate ${c(G)-c(G-E(P)-x)}.$  
We have deleted the vertex~$x$ that costs~$1\over 3$. If~$q\notin \{y_1,y_2\}$ then we have decreased after deleting of~$x$ and~$E(P)$ the cost of three vertices of~$\N_G(x)\setminus\{y\}$  (not more than by~$2\cdot {1\over 4}+{1\over 12}$, since the vertex~$w \in \N_{G}(x)$ does not belong to $S(G)$) and of the vertex~$q$ (by~$1\over 12$).  
If~$q\in\{y_1,y_2\}$ then we have  decreased the cost only of two different from~$y,q$  vertices of~$\N_G(x)$ (not more than by ${1\over 4}+{1\over 12}$).
In both cases we have~${c(G)-c(G-E(P)-x)\le 1}$ and
$$u(G)\ge u(G-E(P)-x) +1 \ge c(G-E(P)-x)+1  +{3\over 2}\ge  c(G)+{3\over 2}.$$

\smallskip
\q{R6.2}. {\it One of the two following conditions holds}:

 --- $d_G(x)=4$ {\it and}~$d_G(w)=3$;

 --- $d_G(x)>4$. 

\noindent As in case~$R5$, we consider the graph~$G'=G\cdot xw$ and the vertex~$x'\in V(G')$, obtained after contracting~$x$ and~$w$. Clearly,~$x'$ is a cutpoint of the graph~$G'$ (it separates a pendant vertex from other vertices). It is easy to see, that $d_{G'}(x')\ge d_{G}(x)\ge  4$: the vertex~$x'$ is adjacent in the graph~$G'$ to all vertices of~$\N_G(x)$ except~$w$ and, in addition, to a pendant vertex from~$\N_G(w)$.

\begin{rem}
\label{rr8}
1) Contracting the edge~$xw$ can decrease (exactly by~1) only degrees of vertices which form a triangle with~$w$ and $x$. All such vertices belong to~$X$ (due to lemma~\ref{lr5}) and have degree at least~4 (otherwise the condition~$2^\circ$ of lemma~\ref{lr4} holds).

2) Hence if $d_G(a)=3$  and~$a\ne w$ then~$d_{G'}(a)=3$.
\end{rem}

If $d_G(w)=3$, then~$c(G') \ge c(G) - {1\over 3}$, as it was proved in~$R5$.

If~$d_G(x)>4$, then all vertices of the set~$X$ have degree at least~5. 
By remark~\ref{rr8} contracting the edge~$xw$ can decrease (exactly by~1) only degrees of some vertices of the set~$X$. Since the degree of any such vertex in the graph~$G$ is at least~5, then its costs in~$G$ and~$G'$ coincide. Thus in this case 

$$c(G)-c(G')=c_G(w)\le {1\over 3}.$$

\smallskip
Remind, that~$y$ is a vertex adjacent to~$x$, such that~$y\notin W$, $d_G(y)=3$, $\N_{G}(y)=\{x,z,z'\}$ and $d_G(z)=d_G(z')=3$ (see the beginning of the case~$R6$). Note, that $\N_{G'}(y)=\{x',z,z'\}$.

{\sf We shall construct in the graph~$G'$ a simple path~$P$  from~$z$ to some vertex~$q$  (where~$q\not\in S(G')\cup \N_G(y)$ or~$q=z'$),  such that all inner vertices of the path~$P$ belong to $S(G')\setminus \N_{G'}(x')$ and the graph~$G'-E(P)-y$ is connected. }

\begin{figure}[!hb]
	\centering
		\includegraphics[width=0.8\columnwidth, keepaspectratio]{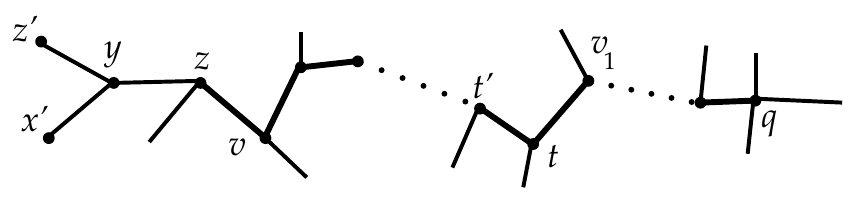}
     \caption{Reduction $R6.2$: the graph~$G'$ and path $P$.}
	\label{figa4}
\end{figure}

Let $v\in \N_{G'}(z)$, $v\not \in\{y,x'\}$.
If the edge~$zv$ is a bridge in~$G'-y$, then~$z$ is a cutpoint of  the graph~$G'-y$. Hence by remark~\ref{rcp5}  the vertex~$z$ is a cutpoint of the graph~$G-y$, then the condition~$1^\circ$ of lemma~\ref{lr4} holds for the vertices~$y$ and~$z$. We obtain a contradiction. 
Then~$zv$ is not a bridge in~$G'-y$.

At the beginning let~$P'$ contain only one edge~$zv$. Note, that the graph ${G'-E(P')-y}$ is connected and~$d_{G'}(v)\ge 3$. 

Let  a path~$P'$ from $z$ to $t$ be built, such that the graph~$G'-E(P')-y$ is connected and~$d_{G'}(t)\ge 3$
(at the beginning we have $t=v$ and all these conditions hold). 
Clearly, $t\in S(G')$ (i.e., $d_{G'}(t)=3$) and $t\ne z'$, otherwise the path~$P=P'$ is what we want. 

Let $\N_{G'}(t)=\{t',v_1,v_2\}$, where~$t'$ is the previous vertex of the path~$P'$ and~$v_1\ne x'$. Then~$d_{G'}(t')=3$. We try to extend the path~$P'$ by one edge.

Let~$tv_1$ be a bridge of the graph~$G'-E(P')-y$. Then~$t$ is a cutpoint of the graph ${G'-E(P')-y-t'}$ (this graph is, obviously, connected, since it is a result of deleting the pendant vertex~$t'$ from the connected graph~$G'-E(P')-y$).
Similarly to item~$6.1$ we obtain, that 
$$u(G)\ge u(G')\ge u(G'-E(P')) \ge u(G'-E(P')-y-t') +2.$$

Let us estimate $c(G)-c(G'-E(P')-y-t')$. As we have proved above, $c(G)-c(G')\le {1\over 3}$.
We have deleted from~$G'$  the vertices~$y$ and~$t'$, that cost together ${1\over 2}$.
Since the cost of vertices of degrees~1 and~3 is the same, then the cost of all vertices of the part~$P'$ besides~$t'$ and~$t$ is the same in both graphs. Thus  deleting~$E(P')$, $y$ and~$t'$ from~$G'$ can   decrease the cost only of 4 vertices: they are two vertices  of~$\N_{G'}(y)\setminus V(P')$  and two vertices of~$\N_{G'}(t')\setminus V(P')$. Since the cost of each of these 4 vertices is decreased by not more than~$1\over 4$, we have
 $$c(G'-E(P')-y-t')\ge c(G')-{1\over 2}- 4\cdot{1\over 4} \ge  c(G)-{11\over 6}.$$ It follows from the induction assumption that
$$u(G)\ge u(G'-E(P')-y-t') +2 \ge c(G'-E(P')-y-t') +2 +{3\over 2}> c(G)+{3\over 2}.$$
In this case the theorem is proved.

Let us consider the remaining case when~$tv_1$  is not a bridge of the graph ${G'-E(P')-y}$. Then~$d_{G'-E(P')}(v_1)\ge 2$,
hence, $v_1\not\in V(P')$ and $d_{G'}(v_1)\ge 3$.
We extend the path~$P'$ by the edge~$tv_1$ and continue our reasonings with the new path and its end~$v_1$. 
Since our graph is finite, the process will be finished and we obtain the desired path~$P$.

Consider the graph~$G'-E(P)-y$. Similarly to proved above, 
$$u(G)\ge u(G')\ge u(G'-E(P)-y)+1$$ and it remains to prove, that 
$$c(G'-E(P)-y)\ge c(G')-{2\over 3}\ge c(G)-1.$$  

We have deleted from~$G'$ the vertex~$y$ that costs~$1\over 4$. If~$q\ne z'$ then we have decreased after deleting~$y$ and~$E(P)$ the cost of two vertices of~$\N_{G'}(y)\setminus\{z\}$  (not more than by~$ {1\over 4}+{1\over 12}$, since the vertex~$x' \in \N_{G'}(y)$ does not belong to $S(G')$) and of the vertex~$q$ (by~$1\over 12$).  
If~$q=z'$ then we have  decreased the cost only of the vertex~$x'$ (not more than by ${1\over 12}$).
In both cases we have
$$c(G')-c(G'-E(P)-x)\le {2\over 3},$$ what was to be proved.

\begin{lem} 
\label{lr8}
Let~$G$ be a connected graph which have pendant vertices, such that~$v(G)\ge 2$ and no reduction rule of~$R1-R6$ can be applied.
Then the graph~$G$ satisfies the three following conditions.

$1^\circ$ Vertices of the set~$X$ are pairwise nonadjacent and have degree at least~$7$.

$2^\circ$  Vertices of the set~$W$ are pairwise nonadjacent and have degree not more than~$4$. 

$3^\circ$  Any pendant vertex of the graph~$G$ is adjacent to a vertex of the set~$W$.
Any vertex of the set~$W$ is adjacent to exactly one pendant vertex of the graph~$G$.
\end{lem}

\begin{proof}
Since~$v(G)\ge 2$, pendant vertices of the graph~$G$ are pairwise nonadjacent.  Hence any pendant vertex is  adjacent to a vertex of the set~$W$.

Since reduction rule~$R6$ cannot be applied,  vertices of the set~$X$ have degrees at least~7. Thus vertices of the set~$X$ are pairwise nonadjacent. 

 By lemma~\ref{lr5} vertices of the set~$W$ are pairwise nonadjacent.  As we know, these vertices have degree at least~3. Hence any vertex 
of the set~$W$ is adjacent to the set~$X$. Since degrees of vertices of the set~$X$ are at least~7  and reduction rule~$R3$ cannot be applied, then degrees of  vertices of the set~$W$ are not more than~4.
\end{proof}

\subsection{Method of dead vertices}
\label{dead}

Let no reduction rule of~$R1-R6$ can be applied.
In this case we shall build the spanning tree in the graph~$G$ with the help of {\it dead vertices method} (see.~\cite{KW,JGM,K2}).

Our modification is different from the standard method: we begin construction of a spanning tree with a forest (not necessary a tree).
On each step we shall add new vertices to our forest or join its connected components.

For an arbitrary graph~$H$ we denote by~$k(H)$  the number of connected components of the graph~$H$.

Let~$F$ be a forest obtained after several steps of construction ($V(F)\subset V(G)$, $E(F)\subset E(G)$). 
We do not delete  edges of~$E(F)$ from our forest on the next steps.

\begin{defin}
We say that a leaf~$x$ of the forest~$F$ is  {\it dead}, if all vertices of~$N_G(x)$ are vertices of the connected component of the forest~$F$,
which contains~$x$.

We denote by~$b(F)$ the number of dead leaves of the tree~$F$.
\end{defin}

\begin{rem}
It is easy to see, that dead leaves remain dead during all next steps of the construction. 
When the algorithm stops and we obtain a spanning tree, all its leaves will be dead.
\end{rem}

For a forest~$F$ we define
\begin{equation}
\label{alpha}
 \alpha(F)= {5\over 6} u(F) + {1\over6}  b(F) - c_G(F) - 2(k(F)-1).
\end{equation}

Set the notations $T=T(G)$ and~$S=S(G)$. In our case~$V(G)=S\cup T$.

\subsubsection{Beginning of the construction}

Let us describe the beginning of construction of a spanning tree. On this step we shall construct a forest~$F^*$ in~$G$ with~$\alpha(F^*)$ large enough, such that~$V(F^*)$ contains all pendant vertices of the graph~$G$. Consider two cases.

\smallskip
\q{B1}. {\it There are no pendant vertices in~$G$.}

\noindent In this case we assume that there is a vertex~$a\in V(G)$ with~${d_G(a)\ge 4}$ (otherwise our theorem is a consequence of the the paper~\cite{KW}). We begin with the base tree~$F^*$, in which the vertex~$a$ is connected with~4 vertices of~$\N_G(a)$. Then $\alpha(F^*)\ge {5\over 6}\cdot 4 - {1\over 3} \cdot 5= {5\over 3}$.

\smallskip
\q{B2}. {\it There exist pendant vertices in~$G$, i.e. $U\ne\varnothing$.}

 \noindent In this case the statements of lemma~\ref{lr8} hold. 
Let~$Y\subset V(G)$ be the set of all vertices adjacent to~$X$ that do not belong to~$W$. Consider the graph~$G^*$ on the vertex set~$W\cup X\cup U\cup Y$, such that~$E(G^*)$ is the set of all edges of~$G$, incident to~$W \cup X$. It is clear from lemma~\ref{lr8} that $G^*$ is a bipartite graph with partition~${W\cup Y}, {X\cup U}$.

Let~$G'$ be a connected component of the graph~$G^*$. We shall construct in~$G'$ a spanning tree~$F'$ with~$\alpha(F')\ge2$.
At first we consider the subgraph~$G''=G'-Y$, let it has~$k$ connected components.
Different connected components of the graph~$G''$ are connected with each other through vertices of the set~$Y$
(see figure~\ref{figa5}).

\begin{figure}[!ht]
	\centering
		\includegraphics[width=\columnwidth, keepaspectratio]{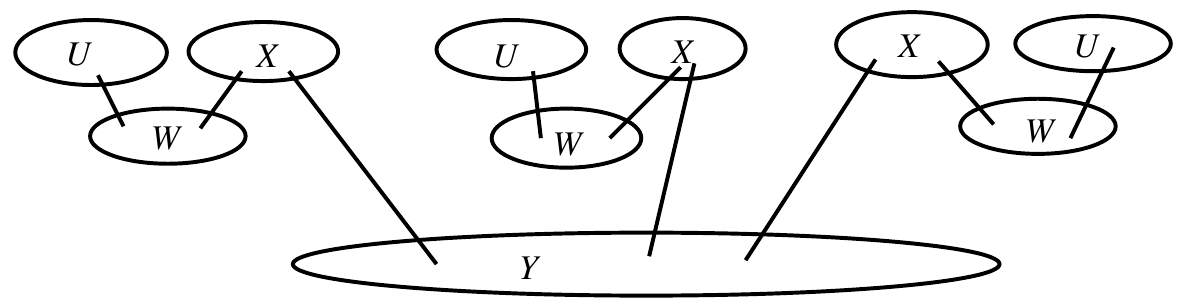}
     \caption{Base $B2$, the graph~$G'$.}
	\label{figa5}
\end{figure}

Let $W'=W\cap V(G')$, $X'=X\cap V(G')$, $Y'=Y\cap V(G')$, $U'=U\cap V(G')$,   $x=|X'|$.
Each vertex of the set~$X'$ is adjacent to at least~7 vertices of~$W'\cup Y'$. Each  vertex of the set~$W'$ is adjacent to exactly one pendant vertex (of the set~$U$) and, consequently, to~2 or~3  vertices of the set~$X'$. Denote by~$W'_2$ and~$W'_3$ subsets of~$W'$, that consists of vertices, adjacent to~2 and~3 vertices of~$X'$ respectively. Clearly,~$W'_2\subset S$ and~$W'_3\subset T$.

Each vertex of the set~$Y'$ is adjacent to a vertex of the set~$X'$, which has degree at least~7. Since reduction rule~$R3$ cannot be applied, 
each vertex of~$Y'$ has degree at most~4, i.e. is adjacent to not more than~4 vertices of~$X'$.
Denote by~$Y'_4$ the subset of~$Y'$, that consists of vertices, adjacent to~4 vertices of~$X'$. Let~$Y'_3=Y'\setminus Y'_4$. Set the notations $w_2=|W'_2|$, $w_3=|W'_3|$, $y_3=|Y'_3|$, $y_4=|Y'_4|$.
Then
\begin{equation}
\label{f1}
 |U'|=w_2+w_3,\quad x \le {2w_2+3w_3+3y_3+4y_4\over 7}. 
\end{equation}

Let us build a spanning forest of~$G'(W'\cup X')$ and after that adjoin to it all vertices of the set~$Y'\cup U'$. We obtain a spanning forest~$F'$ of the graph~$G'$ such that vertices of the set~$Y'\cup U'$ are its leaves. Note, that the number of these leaves is~$w_2+w_3+y_3+y_4$. Clearly, the forest~$F'$ has~$k$ connected components (since the graph~$G''=G'-Y$ has~$k$ connected components). 

We add ${k-1}$ new  edges between~$Y'$ and~$X'$ to the forest~$F'$ such that it becomes a spanning tree~$T'$ of the graph~$G'$. Clearly, after this operation the number of leaves from~$Y'\cup U'$ is decreased  by not more than~$k-1$ and we obtain
$$u(T')\ge w_2+w_3+y_3+y_4-k+1.$$
Let us estimate $\alpha(T')$. We have
$$c_G(T')\le  {1\over 4} \cdot (2w_2+w_3) + {1\over 3} \cdot (w_3 +  x + y_3+y_4).$$
All vertices of the set~$U'$ are dead leaves of the tree~$T'$. Since any vertex of~$Y$ has degree at most~4, then the vertices of~$Y'_4$, that are leaves of the tree~$T'$, are its dead leaves.
Thus, the least possible number of dead leaves of the tree~$T'$ we have in the case, when all leaves of the forest~$F'$, that we loose after joining its connected components, belong to~$Y'_4$. Thus we have 
$$b(T')\ge |U'|+y_4 -k+1 =w_2+w_3+y_4-k+1. $$  
Taking into account inequalities proved above we obtain, that
\begin{equation}
\label{f2}
 \alpha(T')\ge w_2+w_3+y_4+{5\over 6}y_3 - c_G(T')-k+1
\ge w_2\cdot {1\over 2} +w_3\cdot {5\over 12} + y_3\cdot {1\over 2} + y_4\cdot {2\over 3} 
- x\cdot {1\over 3} -k+1. 
\end{equation}

In what follows we  estimate~$k$.  Consider a connected component~$G_0$ of the graph~$G''$. 
Clearly,~$G_0$ contains a vertex~$v\in W'_2\cup W'_3$. 
If~$v\in W'_2$, then~$G_0$ contains at least two vertices of~$X'$. If~$v\in W'_3$ then~$G_0$ contains at least three vertices of~$X'$. Thus, 
\begin{equation}
\label{f3}
  2k \le x. 
\end{equation}

Let~$k_2$ be the number of connected components of~$G''$, which contain two vertices of the set~$X'$. Each such component contains a vertex of~$W'_2$, hence $k_2\le w_2$. Then  
$$x \ge 3(k-k_2)+2k_2= 3k-k_2 \ge 3k-w_2,$$
that we rewrite in the following way:
\begin{equation}
\label{f4}
 3k \le x+w_2. 
\end{equation}

Add inequality~(\ref{f3}) multiplied by~${3\over 2}$ and inequality~(\ref{f4}) and then reduce by~6. We obtain
\begin{equation}
\label{f5}
 k \le {5\over 12}x+{1\over6}w_2. 
\end{equation}

\smallskip
Note, that  $X'\ne\varnothing$, and each vertex of the set~$X'$ is adjacent to at least~7 vertices of~$W'\cup Y'$. Hence
\begin{equation}
\label{f6}
  w_2+w_3+y_3+y_4\ge 7. 
\end{equation}
Consider three cases.

\smallskip
\q{B2.1}. $W'_3=\varnothing$.

\noindent Let us return to inequality~(\ref{f2}). Due to~$w_3=0$ and  inequality~(\ref{f3}) we obtain
$$\alpha(T') \ge w_2\cdot {1\over 2} + y_3\cdot {1\over 2} + y_4\cdot {2\over 3} 
- x\cdot \left({1\over 3} + {1\over 2}\right) +1.$$
Substituting the upper bound on~$x$ from~(\ref{f1}) in the last inequality and taking into account~(\ref{f6}), we obtain
$$ \alpha(T')\ge w_2\cdot {11\over 42}+ y_3\cdot {1\over 7}  + y_4\cdot {4\over  21} +1 
\ge (w_2+y_3+y_4)\cdot {1\over 7} +1 \ge 2, $$
what was to be proved.

\smallskip \goodbreak
\q{B2.2}. $W'_2=\varnothing$, $k=1$.

\noindent  Then by inequality~(\ref{f2}), $w_2=0$ and~$k=1$ we obtain
$$\alpha(T')\ge w_3\cdot {5\over 12} + y_3\cdot {1\over 2} + y_4\cdot {2\over 3} 
- x\cdot {1\over 3}.$$
Substituting the upper bound on~$x$ from~(\ref{f1}) in the last inequality, we obtain
$$ \alpha(T')\ge w_3\cdot {23\over 84}+ y_3\cdot {5\over 14}  + y_4\cdot {10\over  21} 
\ge w_3\cdot {23\over 84} + (y_3+y_4)\cdot {5\over 14}.$$
If $y_3=y_4=0$, then all leaves of the tree~$T'$ are dead, that is $G=G'$.
In this case by inequality~(\ref{f6}) we have~$w_3\ge 7$ and  
$$u(G)\ge u(T')\ge c(G)+{23\over 12},$$ 
the theorem is completely proved.
If $y_3+y_4\ge 1$, then we have $$ \alpha(T')\ge  6\cdot {23\over 84} +  {5\over 14}= 2,$$
that is enough.

\smallskip
\q{B2.3}. $k\ge 2$.

\noindent In the remaining case we assume that~$w_3\ne 0$ (otherwise we get help of the case~$B2.1$).
The graph~$G''$ has at least two connected components and at least one of these components contains vertices of the set~$W'_3$. This component contains at least~3 vertices of the set~$X'$.  Other component contains at least~2 vertices of the set~$X'$.  Hence, $x\ge 5$.
 Substitute in the  inequality~(\ref{f2})   the upper bound on~$k$ from~(\ref{f5}):
$$\alpha(T')\ge w_2\cdot {1\over 2} +w_3\cdot {5\over 12} + y_3\cdot {1\over 2} + y_4\cdot {2\over 3} 
- x\cdot {3\over 4} - w_2\cdot {1\over 6}+1.$$
Due to the upper bound on~$x$ from~(\ref{f1}) and~$x\ge 5$, we have
$$\alpha(T') \ge w_2\cdot {5\over 42} +w_3\cdot {2\over 21} + y_3\cdot {5\over 28} + 
y_4\cdot {5\over 21} +1 \ge x\cdot {2\over 9} +1> 2.$$

\smallskip
We build a spanning tree  in  such way in each connected component of~$G^*$.
{\sf As a result, we have a spanning forest~$F^*$ in the graph~$G^*$. Each connected component~$F'$ of the forest~$F^*$ has~$\alpha(F')\ge 2$. Hence, $\alpha(F^*)\ge 2$.  Note, that~$F^*$ contains all pendant vertices of the graph~$G$.}

\subsubsection{A step of the algorithm}

Let us describe a step~$A$ of the algorithm of construction of spanning tree. Let~$F$ be the forest  we have  before the step~$A$. (Before the first step we have the forest~$F=F^*$ built above.)

Denote by~$\Delta u$ and~$\Delta b$   increase of  number of leaves and dead leaves on the step~$A$ respectively. 
Denote by~$\Delta t$ and~$\Delta s$     number of added on the step~$A$ vertices of the sets~$T$ and~$S$ respectively.
 
Let $F_1$ be the forest obtained after the step~$A$. Set the notation~$\Delta k= k(F)-k(F_1)$.

Let {\it profit} of the step~$A$ be
$$p(A)={5\over 6}\Delta u + {1\over6} \Delta b + 2\Delta k- {1\over3} \Delta t - {1\over4} \Delta s.$$
We will perform only steps with nonnegative profit. 
It follows from the formula~(\ref{alpha}), that $\alpha(F_1)=\alpha(F)+p(A)$.

Let us describe all possible types of a step.
Let~$Z$ denote the set of all vertices which do not belong to~$V(F)$. We call by {\it vertices of level~$1$} all vertices of the set~$Z$, which are adjacent to the set~$V(F)$. For a vertex~$x\in Z$  we denote by~$P(x)$ the set of all vertices from~$V(F)$,   adjacent to~$x$.

\begin{rem}
1)  Note, that any vertex of the set~$Z$ belong to  the set~$S\cup T$  and have degree  at least~3.

2) Estimating  profit of the step we always assume, that all added vertices, which are not known to belong to~$S$, are in the set~$T$.
If some added vertex  belongs to~$S$, then profit of the step will be increased by~$1\over 12$. 
\end{rem}

\smallskip
In what follows we describe several variants of a step of the algorithm.  
{\sf  We try to perform next step of the algorithm in the following way: we pass to the next variant of a step only when all 
previous variants are impossible.} 
 We shall not note this during   description of  steps.

\smallskip
\q{S1}.  {\it There exists an edge~$xy$, which ends belong to different connected components of the forest~$F$.} 

\noindent Then we add the edge~$xy$ to~$F$ and decrease the number of connected components by~1. Thus,~$\Delta k\ge 1$ and~$\Delta u\ge -2$. 
We obtain
$$p(S1) \ge -2\cdot{5\over 6}+2 \ge {1\over 3}.$$

\smallskip \goodbreak
\q{S2}.  {\it There exists a non-pendant vertex~$x\in V(F)$ adjacent to a vertex~${y\in Z}$.} 

\noindent Then we adjoin~$y$ to~$x$. In this case~$\Delta u = 1$, $c_G(y)\le {1\over 3}$ and, hence,
$$p(S2) \ge {5\over 6}-{1\over 3} \ge {1\over 2}.$$

\begin{rem}
Later we assume that edges of~$E(G)\setminus V(F)$ are not incident to non-pendant vertices of~$F$ and cannot join two different connected components of the forest~$F$.
\end{rem}

\smallskip
\q{S3}.  {\it There exists a vertex~$x\in V(F)$, adjacent to two vertices of~$Z$.} 

\noindent We adjoin these two vertices to the vertex~$x$. Then~$\Delta u= 1$,  two added vertex cost not more than~$2\over 3$. Thus
$$p(S3) \ge   {5\over 6}- {2\over 3} \ge {1\over 6}.$$

\smallskip
\q{S4}.  {\it There exists a vertex~$x\in Z$,  adjacent to vertices of two different connected components of the forest~$F$.} 

\noindent Then we join these two components through the vertex~$x$ (we add to the forest the vertex~$x$ and two edges). We obtain~$\Delta u=-2$,  $\Delta k=1$ and, since~$c_G(x)\le {1\over 3}$, then 
$$p(S4) \ge  - 2\cdot {5\over 6}+   2 - {1\over 3} \ge 0.$$

\smallskip \goodbreak
\q{S5}.  {\it There exists a vertex~$x\in Z$, adjacent to $m\ge 3$ vertices of~$V(F)$.} 

\noindent Since the step~$S4$ cannot be performed, the vertex~$x$ is adjacent to three vertices that belong to the same connected component of the forest~$F$. 
We adjoin~$x$ to one of these vertices, two other vertices become dead leaves. Then
$\Delta u=\Delta k=0$, $\Delta b\ge 2$ and, since $c_G(x)\le{1\over 3}$, we obtain
$$ p(S5) \ge   2\cdot{1\over 6} - {1\over 3} =  0. $$

\begin{rem}
 \label{rl1}
In what follows we assume that any pendant vertex of the forest~$F$ is adjacent to exactly one vertex of the set~$Z$ and is not adjacent to vertices of other connected components of~$F$. Any vertex of level~1 is adjacent to not more than two vertices of~$V(F)$ and, if it is adjacent to two vertices of~$V(F)$, they belong to the same connected component of the forest~$F$.
\end{rem}

\q{S6}.  {\it There exists a vertex~$x\in T$ of level~$1$.}

\noindent By remark~\ref{rl1} the vertex~$x$ is adjacent to not more than two vertices of~$V(F)$.
Consider two cases.

\smallskip

\q{S6.1}.  {\it The vertex~$x$ is adjacent to exactly one vertex of~$V(F)$.}

\noindent
Then~$x$ is adjacent to three vertices of the set~$Z$, we adjoin~$x$ and these three vertices to the forest~$F$. 
Thus~$\Delta u=2$. Since the cost of four added  vertices is not more than~$4\cdot {1\over 3}$, we have
$$ p(S6.1) \ge   2\cdot{5\over 6} - {4\over 3} = {1\over 3}. $$

\smallskip \goodbreak
\q{S6.2}.  {\it The vertex~$x$ is adjacent to two vertices of~$V(F)$.}

\noindent 
Then~$x$ is adjacent to two vertices $y_1,y_2\in Z$,  we adjoin~$x,y_1,y_2$ to the forest~$F$ and obtain~$\Delta u=1$. Two vertices of the set~$V(F)$ that are adjacent to~$x$ belong to the same connected component, hence~$\Delta b \ge 1$. Since the cost of three added  vertices is not more than~$3\cdot{1\over 3}=1$, we have
$$ p(S6.2) \ge   {5\over 6} +{1\over 6}- 1 =0.$$

\smallskip

\q{S7}.  {\it There exists a vertex~$x\in S$ of level~$1$, adjacent to exactly one vertex of~$V(F)$.}

\noindent
Then~$x$ is adjacent to two vertices~$y_1,y_2\in Z$. We adjoin~$x,y_1,y_2$ to the forest~$F$ and obtain~$\Delta u=1$. Consider several cases.

\smallskip
\q{S7.1}. $y_1,y_2\in S$.

\noindent  Then the cost of three added  vertices is not more than~$3\cdot{1\over 4}$ and we have 
$$ p(S7.1) \ge   {5\over 6} - 3\cdot{1\over 4} \ge  {1\over 12},$$
that is enough.

\smallskip
\q{S7.2}. $y_1\in T$.

\noindent  Then the cost of three added  vertices is not more than~${1\over 4}+ 2\cdot{1\over 3}={11\over 12}$.

Let~$F_1$ be a connected component of~$F$, which contains a vertex adjacent to~$x$. 
We adjoin to the tree~$F_1$ vertices~$x,y_1,y_2$ and obtain
$$ p(S7.2) \ge   {5\over 6} - {11\over 12} =  -{1\over 12},$$
that is not enough.  Let us analyze what vertices are adjacent to~$y_1$. We consider several cases and continue the step in each case  until its profit becomes nonnegative.

\smallskip
\q{S7.2.1}. {\it The vertex~$y_1$ is adjacent to~$z\in V(F_1)$}.

\noindent Then after the performed step~$z$ becomes a dead leaf of the obtained forest (see figure~\ref{figa6}a), hence~$\Delta b\ge 1$ and we have $ p(S7.2.1) \ge   p(S7.2)  +{1\over 6}={1\over 12},$
that is enough.

\smallskip \goodbreak
\q{S7.2.2}. {\it  The vertex~$y_1$ is adjacent to~$z\in V(F)\setminus V(F_1)$}.

\noindent Then we add to the forest~$F$ the edge~$zy_1$, i.e. we join two connected components of~$F$ (see figure~\ref{figa6}b). We have performed a step~$S1$ in addition to the step performed before, hence
$$ p(S7.2.2) \ge   p(S7.2)  +p(S1)\ge {1\over 4}.$$

\begin{figure}[!ht]
	\centering
		\includegraphics[width=\columnwidth, keepaspectratio]{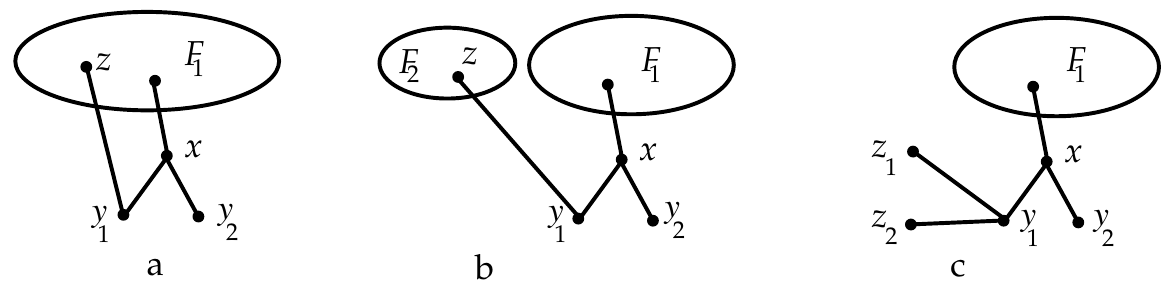}
     \caption{Step $S7.2$.}
	\label{figa6}
\end{figure}

\smallskip \goodbreak
\q{S7.2.3}. {\it The vertex~$y_1$ is not adjacent to~$V(F)$}.

\noindent
Since~$d_G(y_1)\ge 4$, the vertex~$y_1$ is adjacent to two vertices~$z_1,z_2\in Z$. We adjoin~$z_1,z_2$ to~$y_1$,  i.e. we perform a step~$S3$  (see figure~\ref{figa6}c). Hence  $$ p(S7.2.3) \ge p(S7.2) +  p(S3)\ge {1\over 12}.$$

\smallskip

\q{S8}.  {\it There exists a vertex~$x\in S$ of level~$1$, adjacent to two vertices of~$V(F)$.}

\noindent
As it was mentioned before, both vertices of the set~$P(x)$ belong to the same connected component~$F_1$ of the forest~$F$. We add~$x$ to the tree~$F_1$, after that one vertex from~$P(x)$ becomes a dead leaf.
We have~$\Delta s=1$, $\Delta b=1$.

The vertex~$x$ is adjacent to exactly one vertex of the set~$Z$, let it be a vertex~$y$. We add~$y$ to the forest and obtain
$$ p(S8) \ge   {1\over 6} - {1\over 4} -c_G(y) \ge - {1\over 12}-c_G(y).$$
Let us analyze what vertices are adjacent to~$y$. We consider several cases and continue the step in each case until its profit becomes nonnegative.

\smallskip
\goodbreak
\q{S8.1}.  {\it The vertex~$y$ is adjacent to~$V(F)$.}

\noindent Since it is impossible to perform any step of~$S1-S7$ with the vertex~$y$, then~$y\in S$ and $y$ is adjacent to two vertices, that belong to one connected component~$F_2$ of the forest~$F$.  Hence~$c_G(y)={1\over 4}$ and
$$ p(S8) \ge  - {1\over 3}.$$ Consider two cases.

\goodbreak 
\smallskip
\q{S8.1.1}.  $F_2=F_1$.

\noindent
Then both vertices of the set~$P(y)$  and the vertex~$y$ itself became dead leaves (see figure~\ref{figa7}a) 
and we have  
$$ p(S8.1.1) \ge p(S8) + 3\cdot{1\over 6}  ={1\over 6}.$$

\begin{figure}[!ht]
	\centering
		\includegraphics[width=0.6\columnwidth, keepaspectratio]{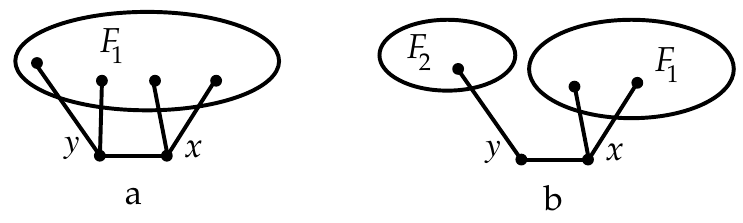}
     \caption{Step $S8.1$.}
	\label{figa7}
\end{figure}

\q{S8.1.2}.  $F_2\ne F_1$.

\noindent Then we add an edge between~$y$ and one vertex from~$P(y)$ and join these two connected components (see figure~\ref{figa7}b). We have performed a step~$S1$, hence
$$ p(S8.1.2) \ge  p(S8) + p(S1) \ge 0.$$

\smallskip

\q{S8.2}.  {\it The vertex~$y$ is not adjacent to~$V(F)$.}

\noindent Then all vertices adjacent to~$y$  except~$x$  are not yet  added to the forest~$F$. Consider two cases.

\smallskip

\q{S8.2.1}.  $y\in T$.

\noindent
Then we can adjoin to the tree~$F_1$ three vertices adjacent to~$y$ (see figure~\ref{figa8}a),  i.e. perform a step~$S3$ and a step~$S2$. Since~${c_G(y)={1\over3}}$,  we obtain $p(S8)\ge -{5\over 12}$ and
$$p(S8.2.1) \ge  p(S8)+ p(S3) +p(S2)\ge {1\over 4}.$$

\smallskip
\q{S8.2.2}.  $y\in S$.

\noindent  In this case we adjoin to the tree~$F_1$ two vertices adjacent to~$y$ (let them be~$z_1$ and $z_2$, see figure~\ref{figa8}b),  i.e. perform a step~$S3$. Since now~${c_G(y)={1\over4}}$ we have $p(S8)\ge -{1\over 3}$ and
$$ p(S8.2.2) \ge  p(S8)  + p(S3) \ge -{1\over 6}.$$
Let us continue case analysis.

\begin{figure}[!ht]
	\centering
		\includegraphics[width=\columnwidth, keepaspectratio]{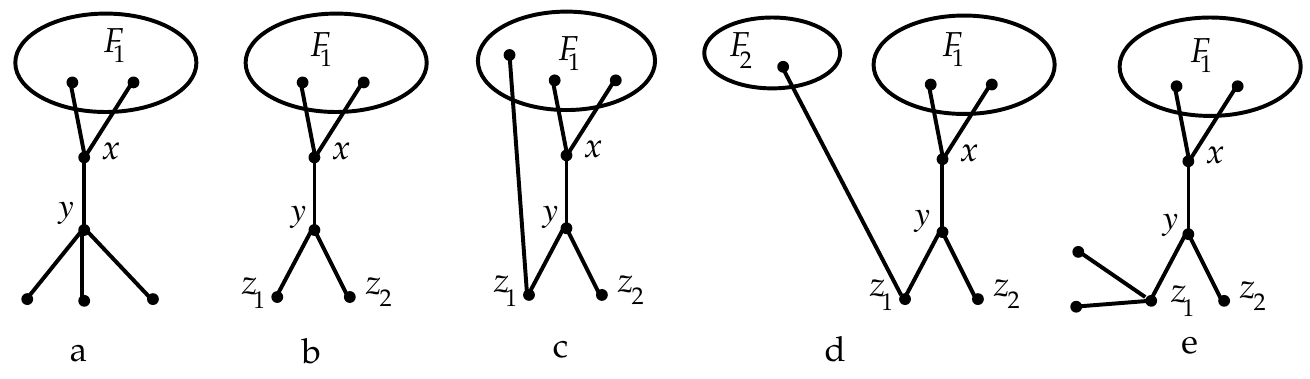}
     \caption{Step $S8.2$.}
	\label{figa8}
\end{figure}

\goodbreak
\q{S8.2.2.1}.  $z_1,z_2\in S$.

\noindent Then two added vertices~$z_1,z_2$ cost together~$2\cdot {1\over 4}$, while there cost was counted as~$2\cdot {1\over 3}$.
Hence $$ p(S8.2.2.1) \ge   p(S8.2.2) + {1\over 6} \ge 0,$$ that is enough.

\smallskip
\q{S8.2.2.2}. {\it The vertex~$z_1\in T$ is adjacent to~$V(F)$.}

\noindent 
If~$z_1$ is adjacent to the tree~$F_1$ (see figure~\ref{figa8}c),  then one extra dead leaf is added to the parameters of step~$S8.2.2$ and we obtain $$ p(S8.2.2.2) \ge   p(S8.2.2) + {1\over 6} \ge 0.$$ 

Let $z_1$ be adjacent to the tree~$F_2\ne F_1$. Then we add to the forest an edge joining~$F_2$ to~$z_1$ (see figure~\ref{figa8}d). We have decreased  the number of connected components of the forest~$F$, i.e. have performed a step~$S1$. Hence
$$ p(S8.2.2.2) \ge   p(S8.2.2) + p(S1)\ge {1\over 6}.$$

\smallskip
\q{S8.2.2.3}. {\it The vertex~$z_1\in T$ is not adjacent to~$V(F)$.}

\noindent  
Since $d_G(z_1)\ge 4$, the vertex~$z_1$ is adjacent to two vertices which are not yet added to the forest~$F$.
We add these two vertices to~$F$ (see figure~\ref{figa8}e), i.e. perform a step~$S3$. Hence
$$ p(S8.2.2.3) \ge   p(S8.2.2) + p(S3)\ge 0.$$

\smallskip
Next lemma will finish the proof of theorem~\ref{u134}.

\begin{lem}
\label{lfa}
Assume that no  step described above can be  performed. Then~$F$ is a spanning tree of the graph~$G$, such that $u(F)\ge c(G)+{5\over 3}$.

\begin{proof}
Assume that~$F$ is not a spanning subgraph of~$G$. Then~$Z\ne \varnothing$ and  there exists a vertex~$a\in Z$ adjacent to~$F$. By the construction we have~$d_G(a)\ge 3$. It is easy to see that then one of the steps can be performed. We obtain a contradiction. Hence~$F$ is a spanning forest of the graph~$G$.

Assume that the graph~$F$ is disconnected. Then it have two connected components which are joined by an edge of~$E(G)\setminus E(F)$ and we can perform a step~$S1$. Hence,~$F$ is connected, i.e. it is a spanning tree of~$G$. Clearly, all leaves of the tree~$F$  are dead and~$k(F)=1$, hence
$$u(F) ={5\over 6}u(F)+{1\over 6}b(F) \ge c_G(F)+ {5\over 3}.$$
\end{proof}
\end{lem}

\section{Extremal examples.}

Next lemma will help us to glue big extremal examples for the bound of theorem~\ref{u134} from small ones.
The main restriction to the objects of gluing is that they must have pendant vertices. The ideas of the proof of lemma~\ref{tool2} are the same as in the paper~\cite{KB}, but with another cost function~$c(G)$. The proof will be correct for any cost function which count pendant vertices with the coefficient~${1\over 4}$.  This lemma can be interpreted as the second item of lemma~\ref{tool}.

\begin{lem}
\label{tool2} Let~$G_1$ and~$G_2$ be connected graphs with~$v(G_1)>2$, $v(G_2)>2$,  $V(G_1)\cap V(G_2)=\varnothing$ and pendant vertices~$x_1$ and~$x_2$. Let~$G$  be the result of gluing the graphs~$G_1$ and~$G_2$ by vertices~$x_1$ and~$x_2$ (let these vertices be glued together into the vertex~$x$)  and contracting one bridge incident to~$x$. Let  $u(G_1)= c(G_1)+{3\over 2}$ and $u(G_2)= c(G_2)+{3\over 2}$. Then $u(G)= c(G)+{3\over 2}$. 

\begin{proof}
Note, that~$v(G)= v(G_1)+v(G_2)-2$. Indeed, two vertices~$x_1$ and~$x_2$ were glued together into the vertex~$x$. After that we have contracted an edge and, hence, have reduced the number of vertices by~1.  As a result of contraction, the vertex~$x$ of degree~2 has disappeared. 
Thus,   vertices of the graph~$G$  are all different from~$x_1$ and~$x_2$ vertices of  the graphs~$G_1$ and~$G_2$. Moreover, any different from~$x_i$ vertex~$y\in V(G_i)$ has~$d_G(y)=d_{G_i}(y)$. Since the vertices~$x_1\in V(G_1)$ and~$x_2\in V(G_2)$   do not belong to~$V(G)$ and~$c_{G_1}(x_1)=c_{G_2}(x_2)={1\over 4}$,  then we have $c(G)= c(G_1)+ c(G_2) -2\cdot {1\over 4}$.

Let us apply lemma~\ref{tool} and write the chain of equalities:
$$u(G)=u(G_1)+u(G_2)-2 =c(G_1)+c(G_2)+ 1 = c(G) +{3\over2}.$$
\end{proof}
\end{lem}

It remains to construct a graph for which the bound of theorem~\ref{u134} is tight, that contains nat least two pendant vertices.

\begin{figure}[!ht]
	\centering
		\includegraphics[width=\columnwidth, keepaspectratio]{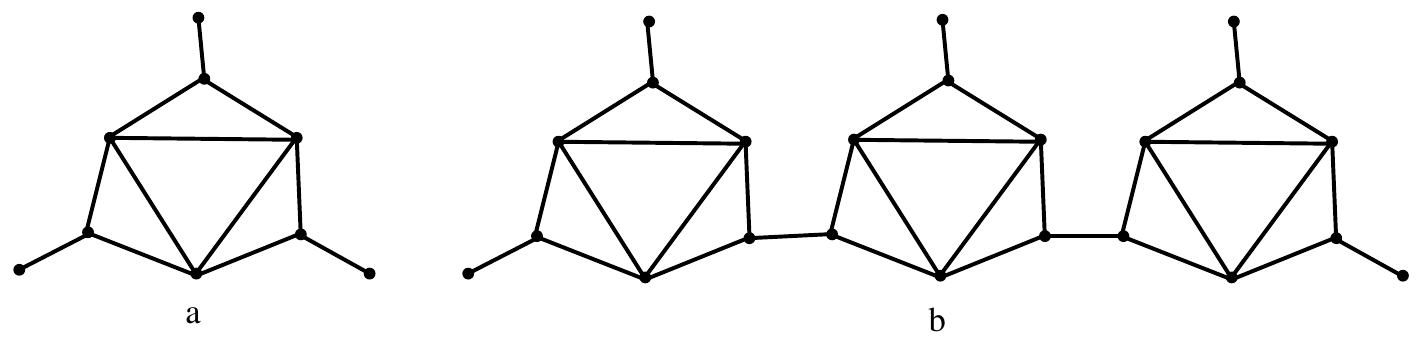}
     \caption{Extremal examples.}
	\label{figa9}
\end{figure}

We show such a graph~$G$ on figure~\ref{figa9}a. This graph contains three vertices of each of  degrees~1, 3 and~4, hence
$$ t=3, \quad  s=6, \quad \mbox{and} \quad c(G)=3\cdot {1\over 3}+ 6\cdot {1\over 4}={5\over 2}. $$
Clearly,~$u(G)=4$. Indeed, the set of leaves of any spanning tree of~$G$ contains three pendant vertices of~$G$ and can contain  at most one vertex of degree~4. Thus, $u(G)=4=c(G)+{3\over 2}$.
Then by lemma~\ref{tool2} we can construct arbitrary long chains of graphs~$G$ (see figure~\ref{figa9}b) and the lower bound on the number of leaves of theorem~\ref{u134} is tight for all such graphs.

\end{document}